\documentclass[]{interact}
\usepackage[misc]{ifsym}
\usepackage{hyperref}
\usepackage{amsmath}
\usepackage{amssymb}
\usepackage{amsthm}
\usepackage{mathrsfs}
\usepackage{epstopdf}
\usepackage[caption=false]{subfig}

\usepackage[numbers,sort&compress]{natbib}
\bibpunct[, ]{[}{]}{,}{n}{,}{,}

\theoremstyle{plain}
\newtheorem{theorem}{Theorem}[section]
\newtheorem{lemma}[theorem]{Lemma}

\newtheorem{proposition}[theorem]{Proposition}
\newtheorem{assumption}{Assumption}

\theoremstyle{definition}
\newtheorem{definition}[theorem]{Definition}
\newtheorem{example}[theorem]{Example}

\theoremstyle{remark}
\newtheorem{remark}{Remark}

\begin{document}

\title{On stationarity properties of generalized Hermite-type processes}

\author{
\name{Illia~Donhauzer\textsuperscript{a}\thanks{I.~Donhauzer. Email:I.Donhauzer@latrobe.edu.au} and Andriy~Olenko\textsuperscript{a}\thanks{\Letter \ A.~Olenko. Email:a.olenko@latrobe.edu.au}}
\affil{\textsuperscript{a} La Trobe University, Melbourne, Australia}
}

\maketitle

\begin{abstract}
The paper investigates properties of generalized Hermite-type processes that arise in non-central limit theorems for integral functionals of long-range dependent random fields. The case of increasing multidimensional domain asymptotics is studied.  Three approaches to investigate properties of these processes are discussed. Contrary to the classical one-dimensional case, it is shown that for any choice of a multidimensional observation window the generalized Hermite-type process has non-stationary increments.
\end{abstract}


\section{Introduction}

This paper investigates properties of limit processes in  non-central limit theorems for nonlinear integral functionals \cite{ivanov1989statistical}, \cite{taqqu1979convergence}. The structure of the increments of the limit processes for different integration sets is of the interest. While stationarity of increments of self-similar stochastic processes is well investigated, the case of random fields attracted increasing attention only recently, see, for example, the discussion in~\cite{fu2019stable, makogin2015}.  This paper investigates the class of generalized Hermite-type processes obtained via asymptotics of nonlinear transformations of long-range dependent random fields.

Dobrushin and Major \cite{dobrushin1979non} considered the nonlinear functionals 
$$Y_m^N =\frac{1}{A_N}\sum_{j=(m-1)N}^{mN-1}G(\xi_j)$$ of long-range dependent Gaussian random sequences $\{\xi_j, \ j \in \mathbb{N}\}$ with correlation functions of the form $r(j)=\frac{L(j)}{j^\alpha}, \ \alpha \in (0,1),$ where $A_N$ are normalising coefficients and $L(\cdot)$ is a slowly varying function at the infinity. It was shown that the asymptotic behavior of $Y_m^N, \ N \to \infty,$ depends on the Hermite rank $\kappa$ of the non-random function $G$ and in a general case the asymptotic distribution is not Gaussian. In \cite{dobrushin1979non} and the more general case \cite{bai2014generalized}, it was proved that, if $N \to \infty,$ then under certain conditions the finite-dimensional distributions of $Y_m^N$ converge to finite-dimensional distributions of Hermite processes defined by Wiener-It\^{o} integrals \cite{dobrushin1979gaussian}. In recent years, the Malliavin calculus approach was used to obtain such results under rather general assumptions, see  \cite{azmoodeh2019} and the references therein. In \cite{taqqu1979convergence} and recently \cite{bai2015functional} a continuous version of the problem was considered by changing summation by integration and Gaussian sequences $\{\xi_j, \ j\in \mathbb{N}\}$ by a Gaussian random processes $\{\xi(x), \ x\geq 0\}.$ The averaging of a nonlinear transformation of the long-range dependent random process $\xi(x)$ over increasing intervals of $\mathbb{R}_{+}$ was considered. It was demonstrated that the finite-dimensional distributions of the integrated processes do not converge (in a general case) to the Brownian motion because of the long-range dependence of the stochastic process $\xi(x).$ Taqqu proved that the limit process $Y(t), \ t \in [0,1],$ is the fractional Brownian motion if and only if the Hermite rank $\kappa=1,$ and $Y(t)$ is non-Gaussian if $\kappa>1$. A representation of the limit process in terms of Wiener-It\^{o} integrals was obtained. The limit process $Y(t)$ belongs to the class of self-similar processes, and $Y(t)$ depends only on the Hermite rank of the function $G$ and the parameter $\alpha,$ which is the rate of a hyperbolic decay of the correlation function at the infinity. 

In \cite{dobrushin1979non} these one-dimensional results were also generalized to the multidimensional case with the summation over integer grid points of multidimensional parallelepipeds. In \cite{anh2019lse, ivanov1989statistical, leonenko2014sojourn}  the corresponding continuous multidimensional case was considered, when the integration of long-range dependent homogeneous isotropic random fields $\xi(x),$ $x \in\mathbb{R}^n,$ is over homothetic transformations $\Delta(r t^{1/n}), \ t\in[0,1], \ r\to\infty,$ of multidimensional observation windows $\Delta \in \mathbb{R}^n$. Similar to the one-dimensional case, it was shown that for $\kappa>1$ the limit process is not Gaussian and is self-similar. It was demonstrated that the limit processes are different for different $\Delta.$ We call them as generalized Hermite-type processes.

This paper studies the limit processes $Y(t).$  It is well-known that in the one-dimensio\-nal case $n=1$ these limit processes have stationary increments. It is commonly assumed that $Y(t)$ possesses the same property for the case of integral functionals over multidimensional windows. However, we prove that in the multidimensional case $n> 1$ the limit processes always have non-stationary increments. 

This paper shows interesting relationships between the increments of the limit processes and geometric probabilities. Crofton's mean value formula \cite{baddeley1977integrals, kendall1998stochastic} for an average function of  distances of points inside a growing domain is used. This formula is an important tool that finds various applications in differential geometry, shape analysis, spatial statistics, just to mention few, see \cite{baddeley1977integrals, kendall1998stochastic, nickolas2011distance} and the references therein. Crofton's formula connects a differential of a functional of the average distance $M(x_1,..., x_k)$  between $k$ uniformly distributed points $x_i, i = 1,...,k,$ inside the growing domain $\Delta(t)$ and geometric properties of the surface of this domain $\partial \Delta(t)$. 

Moreover, it will be seen that variances of increments of the limit processes can be defined as integrals of positive-definite functions. Such integrals are of the interest in analysis as different applications require optimal estimators of these integrals,   see \cite{gaal2018integral, phillips2019extension}. The integral settings let use more general classes of positive definite functions than the classical definition based on quadratic forms and finite sums, see \cite{phillips2019extension}. 

To prove the results the paper employs three different methodologies based on sto\-chastic integral representation of the Hermite-type processes, Crofton's mean value formula and integrals of potential kernels. The obtained results show interesting links between stochastics and differential geometry and can be used in other applications.

The paper is organized as follows. Section 2 provides the main definitions and notations required in the following sections. Non-central limit theorems for random processes and properties of the limit processes are given in Section 3. Section 4 provides non-central limit theorem for random fields and properties of the corresponding limit process. Numerical studies confirming the obtained theoretical results are presented in Section 5.

In what follows we use the symbol $C$ to denote constants which are not important for our discussion. Moreover, the same symbol $C$ may be used for different constants appealing in the same proof. By $||\cdot||$ we denote the norm in $n$\mbox{-}dimensional Euclidean space, and $|\cdot|$ stands for the Lebesgue measure of sets in the same space. We use the notation $A+y = \{ x \in \mathbb{R}^n: x = z + y, z \in A \}.$

\section{Premilinaries}

This section states the main definitions and notations used in this paper.

\begin{definition} A random field $\xi(x), \ x \in \mathbb{R}^n ,$ is called strictly homogeneous and isotro\-pic, if finite-dimensional distributions of $\xi(x)$ are invariant with respect to the groups of motion and rotation transformations 
$$P\big(\xi(x_1)<a_1,...,\xi(x_k)<a_k\big) = P\big(\xi(Ax_1+h)<a_1,...,\xi(Ax_k+h)<a_k\big)$$  for all rotation transformations $A,$ vectors $h ,x_1,x_2,..,x_k\in\mathbb{R}^n$ and $a_1,a_2,..,a_k\in\mathbb{R}.$
\end{definition}

\begin{definition} A random field $\xi(x),\ x\in\mathbb{R}^n,$ is self-similar with parameter $H$ if $\xi(ax) \overset{d}{=} a^H\xi(x),$  where $\overset{d}{=}$ denotes the equality of finite-dimensional distributions.
\end{definition}

\begin{definition} A measurable function $L:(0,\infty) \to (0,\infty)$ is called slowly varying at the infinity if for all $\lambda>0$
$$\lim_{t \to + \infty} \frac{L(\lambda t )}{L(t)} = 1.$$  
\end{definition}

\begin{definition} The function
$$H_m(u)=(-1)^\kappa e^{u^2/2}\frac{d^m}{du^m}e^{-\frac{u^2}{2}}$$ is a Hermite polynomial of order $m$. 

The first few Hermite polynomials are 
$H_0(u) = 1, \ H_1(u) = u, \ H_2(u) = u^2-1.$
\end{definition}

Let $\phi(x) = exp\{-u^2/2\}/\sqrt{2\pi}$ be a probability density of the standard Gaussian random variable. Denote by $L_2(\mathbb{R}, \phi(u)du)$ a Hilbert space of Lebesgue measurable functions $G:\mathbb{R} \to \mathbb{R}$ such that $\int_{\mathbb{R}}G^2(u)\phi(u)du<\infty.$ 

It is known that the Hermite polynomials form a complete orthogonal system in the space $L_2(\mathbb{R}, \phi(u)du)$, i.e.
$$\int_{\mathbb{R}}H_{m_1}(u)H_{m_2}(u)\phi(u)du = \delta_{m_1}^{m_2}m_1!, \ m_1,m_2 = 0,1,...,$$ where $\delta_{m}^q$ is a Kronecker delta.

\begin{definition} The Hermite rank of a function $G \in L_2(\mathbb{R}, \phi(u)du)$ is the index $\kappa$ of the first non-zero coefficient $a_\kappa$ in the expansion of $G$ into the Hermite polynomials
$$G(u) = \sum_{m=\kappa}^\infty a_mH_m(u).$$
\end{definition}

\begin{assumption}
\label{assump1}
 Let $\xi(x), \ x \in \mathbb{R}^n,$ be a mean square continuous homogeneous isotro\-pic Gaussian random field with $E\xi(x)=0$ and a covariance function 
$$B(r) = E\big(\xi(0) \xi(x)\big) = \frac{L(||x||)}{||x||^{\alpha}},$$ where $\alpha \in (0,n),\ r=||x||$ and $L(\cdot)$ is a slowly varying function at the infinity. 

\end{assumption}

If Assumption~\ref{assump1} is satisfied, then by properties of the Hermite polynomials of Gaussian random fields 
$$EH_{m_1}(\xi(x)) = 0,$$
$$EH_{m_1}(\xi(x_1))H_{m_2}(\xi(x_2)) = \delta_{m_1}^{m_2}m_1!B^{m_1}(||x_1-x_2||)$$ $$ = \frac{\delta_{m_1}^{m_2}m_1! L^{m_1}(||x_1-x_2||)}{||x_1-x_2||^{m_1\alpha}}, \ x_1,x_2\in\mathbb{R}^n.$$

It follows from the Abelian and Tauberian theory, see \cite{leonenko2013tauberian}, that Assumption \ref{assump1} on the covariance function can be replaced by analogous conditions on the spectral density of $\xi(x).$ For example, if $L(\cdot)\equiv const,$ one can use

\begin{assumption}
\label{assump2}
 Let $\xi(x), x \in \mathbb{R}^n,$ be a mean square continuous homogeneous isotro\-pic Gaussian random field with the spectral density $f(||\lambda||) = h(||\lambda||)/||\lambda||^{n-\alpha}, \ \lambda \in \mathbb{R}^n,$ such that it holds $h(\rho)/\rho^{1-\alpha} \in L_1(\mathbb{R}_+),\ \rho\in[0,\infty),$ $\alpha \in (0,n)$ and $h(\rho)$ is a continuous function in a neighborhood of the origin, $h(0) \neq 0$ and  $h(\rho)$ is bounded on $\mathbb{R}_+.$  
\end{assumption}

\begin{lemma}\label{lemma1} {\rm \cite[Lemma 2.10.1]{ivanov1989statistical},  \cite[\S 5]{leonenko2013tauberian}}  If Assumption {\rm\ref{assump2}} is satisfied, then the correlation function of the random field $\xi(x), \ x\in\mathbb{R}^n,$ has the following asymptotic behavior
$$B(r) \sim \frac{h(0)c_1(n, \alpha)(1+O(1))}{r^\alpha},  \ r \to \infty,$$
where $c_1(n, \alpha) = 2^{\alpha}\pi^{n/2}\Gamma(\alpha/2)/\Gamma((n-\alpha)/2).$
\end{lemma}

Let the function $v(x,t),$ $x\in \Delta(t),$ $t>0,$ represent the velocity of change of the sets $\Delta(t)$ at point $x$ at moment $t.$ $\partial \Delta(t)$ will denote the boundary of the set $\Delta(t).$

The following result is the celebrated Crofton's formula.

\begin{theorem} {\rm{\cite[Theorem 1.7]{kendall1998stochastic}}}\label{croft}
Let $\{\Delta(t) \}, \ t \in [0,\infty),$ be a family of compact subsets of $\mathbb{R}^n$ that are smoothly changing in the sense that the graph $\Gamma = \{(x,t): x \in \Delta(t) \}$ is a twice continuously differentiable $n+1$ dimensional embedded manifold with a boundary in $\mathbb{R}^n\times\mathbb{R}.$  Consider 
$M(t) = Ef(X_1,...,X_m),$ where $X_1,...,X_m$ are independent uniformly distributed in $\Delta(t)$ random points and $f:\mathbb{R}^m\to\mathbb{R}$ is a symmetric function of its arguments. 

Then almost everywhere $M$ has the derivative
$$\frac{d}{dt}M(t) = n\frac{\frac{d}{dt}|\Delta(t)|}{|\Delta(t)|}(M_1(t)-M(t)),$$  where $M_1(t)=Ef(Y,X_2,...,X_m),$  and $Y$ is distributed on $\partial \Delta(t)$ with a density proportional to $v(x,t).$ 
\end{theorem}

\section{One-dimensional case of functionals of stochastic processes}

This section reviews non-central limit theorems for stochastic processes and discusses stationarity properties of the corresponding limit processes $Y(t), \ t\in[0,1].$

Let a function $G:\mathbb{R} \to \mathbb{R}$ has Hermite rank $\kappa$ and satisfy conditions $EG(X)=0,$ $EG^2(X)< \infty,$ where $X$ is the standard Gaussian random variable.

In \cite{taqqu1979convergence}, for a suitably chosen function $e(u) , \ u\in\mathbb{R},$ such that $e(u) \sim u^{H_0-3/2}L(u),\ u\to\infty,$ where $L(u)$ is slowly varying function at the infinity and $1-\frac{1}{\kappa}<H_0<1$, the process $X(s),\ s \in\mathbb{R},$ was defined by

$$X(s) =\frac{1}{\sigma} \int_{\mathbb{R}} e(s-\xi)dW(\xi),\ s\in \mathbb{R},$$ where $\sigma^2 = \int_{\mathbb{R}}e^2(u)du$ and $W(\cdot)$ is the standard Gaussian white noise measure satisfying $EW(A)=0$ and
$EW^2(A)=|A|$ for Borel sets $A$ of finite Lebesgue measure $|A|$. The process $X(s)$ is 
Gaussian, stationary, satisfies $EX(s)=0,$ $EX^2(s)= 1$ and 

$$EX(s)X(s+x)\sim Cx^{2H_0-2}L^2(u)$$ as $x\to \infty,$ where $C$ is a positive constant.

Then the following limit theorem holds true.

\begin{theorem}{\rm{\cite[Theorem 5.5]{taqqu1979convergence}}} Let $d(r) \sim E\left(\int_{0}^{r} G(X(s))ds\right)^2,$ $r\to\infty.$ Then, for $r\to\infty,$ the finite-dimensional distributions of the process 
$$\frac{1}{d(r)}\int_{0}^{tr}G(X(s))ds,\ t\in [0,1],$$ 
converge weakly to the finite-dimensional distributions of the process 
$$Y(t) = K(\kappa, H_0)\int_{\mathbb{R}}\int_{-\infty}^{\xi_1}...\int_{-\infty}^{\xi_{\kappa-1}}\int_0^t\prod_{i=0}^\kappa\bigg( (s-\xi_i)^{H_0-3/2}I(\xi_i < s) \bigg)ds dW(\xi_{\kappa})...dW(\xi_1),$$  where  $K(\kappa, H_0) $ is a constant.

\end{theorem}
The process $Y(t), \ t\in[0,1],$ is called the Hermite process. If the Hermite rank $\kappa=1,$ then $Y(t)$ is the fractional Brownian motion \cite{mishura2008stochastic}.

It is well-known that the limit process has the following property for any $\kappa \geq 1,$ but for the completeness of the exposition we will present the proof which uses the approach that is different from the main result in the following section.

\begin{proposition} The process $Y(t), \ t\in [0,1],$ is self-similar with stationary increments.
\end{proposition}

\begin{proof}

The Gaussian white noise $W(\cdot)$ is a self-similar random measure. Thus, the process $Y(t), \ t\in[0,1],$ is self-similar for all $\kappa$ and $H_0$. Moreover, the stationarity of increments for all $\kappa$ and $H_0$ follows from the transformations below.

Without loss of generality let $t_1 > t_2,$ $t_1,t_2\in[0,1].$  Then $Y(t_1) - Y(t_2)$ equals in distribution to
$$ K(\kappa,H_0) \int_{-\infty}^{\infty} \int_{-\infty}^{\xi_1}...\int_{-\infty}^{\xi_{\kappa-1}}\int_{0}^{t_1}\prod_{i=1}^{\kappa}\bigg((s-\xi_i)^{H_0-\frac{3}{2}}I(\xi_i<s)\bigg)dsdW(\xi_{\kappa})...dW(\xi_1)  $$
$$-K(\kappa,H_0) \int_{-\infty}^{\infty} \int_{-\infty}^{\xi_1}...\int_{-\infty}^{\xi_{\kappa-1}}\int_{0}^{t_2}\prod_{i=1}^{\kappa}\bigg((s-\xi_i)^{H_0-\frac{3}{2}}I(\xi_i<s)\bigg)dsdW(\xi_{\kappa})...dW(\xi_1)$$ 
$$ {\buildrel d \over =}K(\kappa,H_0) \int_{-\infty}^{\infty} \int_{-\infty}^{\xi_1}...\int_{-\infty}^{\xi_{\kappa-1}}\int_{t_2}^{t_1}\prod_{i=1}^{\kappa}\bigg((s-\xi_i)^{H_0-\frac{3}{2}}I(\xi_i<s)\bigg)dsdW(\xi_{\kappa})...dW(\xi_1).$$

By the change of variables $s' = s - t_2, \ \xi_1'=\xi_1-t_2$ we obtain
$$Y(t_1) - Y(t_2) {\buildrel d \over =} K(\kappa,H_0) \int_{-\infty}^{\infty} \int_{-\infty}^{\xi_1'+t_2}...\int_{-\infty}^{\xi_{\kappa-1}}\int_{0}^{t_1-t_2}(s'-\xi_1')^{H_0-\frac{3}{2}}I(\xi_1'<s')$$
$$ \times \prod_{i=2}^{m}\bigg((s'+t_2-\xi_i)^{H_0-\frac{3}{2}}I(\xi_i<s'+t_2)ds'\bigg)dW(\xi_{\kappa})...dW(\xi_1^{'}).$$

Changing variables as $\xi_i'=\xi_i-t_2, i = 2,...\kappa,$ by induction the last expression equals
$$ K(\kappa,H_0) \int_{-\infty}^{\infty} \int_{-\infty}^{\xi_1'}...\int_{-\infty}^{\xi_{\kappa-1}'}\int_{0}^{t_1-t_2}\prod_{i=1}^{\kappa}(s'-\xi_i')^{H_0-\frac{3}{2}}I(\xi_i'<s')ds'dW(\xi_{\kappa-1}')...dW(\xi_{1}'),$$ which shows that the process $Y(t), \ t\in[0,1],$ has stationary increments.
\end{proof}

Thus, the averaging of nonlinear transformations of long-range dependent Gaussian random processes over the homothetic intervals  $[0,rt] \subset \mathbb{R}_{+}$ leads to limit processes with stationary increments.

\section{Multidimensional case of functionals of random fields}

The aim of this  section is to demonstrate that the result of Section 3 is not true for the averaging over multidimensional sets.

Let $\xi(\omega,x):\Omega \times \mathbb{R}^n\to\mathbb{R}$ be a measurable Gaussian long-range dependent homogeneous isotropic random field.

We will consider asymptotics of the nonlinear functionals 
\begin{equation}
\label{eq}
\int_{\Delta(rt^{1/n})}G(\xi(x))dx, \ \ r\to \infty, \ t\in[0,1],
\end{equation} where $\Delta(rt^{1/n})$  is a homothetic transformation with parameter $rt^{1/n}$ of a simply connected $n\mbox{-}$dimensional  compact set $\Delta\subset\mathbb{R}^n$ containing the origin with the Lebesgue measure $|\Delta| > 0$.  Note that the integral \eqref{eq} exists with probability 1, see Theorem 1.1.1 in \cite{ivanov1989statistical}. 

Let $\int_{\mathbb{R}^{n\kappa}}^{'}$ denote the Wiener-It\^{o} multidimensional stochastic integral, see \cite{dobrushin1979gaussian}.

\begin{theorem} {\rm{\cite[Theorem 2.10.2]{ivanov1989statistical}, \cite[Theorem 5]{leonenko2014sojourn}}} \label{leonenko} Let Assumption {\rm{\ref{assump1}}} or {\rm{\ref{assump2}}} be satisfied and $\alpha \in (0, n/\kappa).$  Then, if $r \to \infty$ the finite-dimensional distributions of the process
$$Y_{r}(t) = \frac{\int_{\Delta(r t^{1/n})}H_\kappa(\xi(x))dx}{r^{n-\kappa\alpha/2}\sqrt{c_2(n,\kappa,\alpha,\Delta)}L^{\kappa/2}(r)}, \ t \in [0,1],$$
converge weakly to the finite-dimensional distributions of the process
$$Y(t) = \int_{\mathbb{R}^{n\kappa}}^{'}\prod_{j=1}^\kappa ||\lambda||^{(\alpha - n)/2}\int\displaylimits_{\Delta(t^{1/n})}e^{i(\lambda^{(1)}+...+\lambda^{(\kappa)},x)}dx\prod_{j=1}^\kappa W(d\lambda^{j}), \ t \in[0,1],$$where $c_2(n,\kappa,\alpha, \Delta) = c_1^\kappa(n, \alpha)\kappa!\int_{\Delta}\int_{\Delta}||x-y||^{-\kappa \alpha}dx dy.$  
\end{theorem}

\begin{remark}

By the self-similarity of the Gaussian white noise $W(d(ax)) \overset{d}{=} a^{n/2}W(dx)$, we obtain that the  process $Y(t), \ t\in[0,1],$ is self-similar with parameter $1-\alpha \kappa/2n.$ It means that the self-similarity of the limit processes preserves in the multidimensional case.

\end{remark}

\begin{lemma}\label{lemma2} The variance of  increments of the limit process $Y(t), \ t\in[0,1],$ has the representation 
$$Var(Y(t+h)-Y(t)) = \frac{\kappa!|\Delta|^2h^2}{c_{2}(n,\kappa,\alpha, \Delta)}E(||U-V||^{-\kappa\alpha}),$$ where $h\in[0,1-t],$ $U,V$ are independent uniformly distributed  random vectors in the set $\Delta((t+h)^{1/n})\setminus \Delta(t^{1/n}).$
\end{lemma}
\begin{proof} By Theorem \ref{leonenko} variances of increments of $Y_r(t)$ converge to variances of increments of $Y(t),$ when $r\to\infty.$ Hence, we get
$$Var(Y(t+h)-Y(t)) = \lim_{r \to \infty}Var(Y_{r}(t+h)-Y_{r}(t))$$
$$= \lim_{r \to \infty}\frac{1}{r ^{2n-\kappa\alpha}c_{2}(n,\kappa,\alpha,\Delta)L^\kappa(r)} Var \bigg(\int\displaylimits_{\Delta(r (t+h)^{1/n})\setminus \Delta(r t^{1/n})}H_\kappa(\xi(x))dx\bigg)$$
$$=\lim_{r \to \infty}\frac{\kappa!}{r^{2n-\kappa\alpha}c_{2}(n,\kappa,\alpha,\Delta)L^\kappa(r)}\iint\displaylimits_{\big(\Delta(r (t+h)^{1/n})\setminus \Delta(r t^{1/n})\big)^2}B^\kappa(||x-y||)dxdy,$$ where $\iint\displaylimits_{A^2}$ denotes the double integral $\int_A\int_A.$

By Lemma \ref{lemma1}, changing variables $x = r(t+h)^{1/n}\widetilde{x}, \ y = r(t+h)^{1/n}\widetilde{y}$  and using the property of integrals of slowly varying functions $\int_A f(s)L(rs)ds\sim L(r)\int_A f(s)ds,$  as $r\to\infty$, see Theorem 2.7 in \cite{seneta1976functions}, we obtain
$$Var(Y(t+h)-Y(t))= \lim_{r \to \infty}\frac{\kappa!r^{2n}}{r^{2n-\kappa\alpha}c_{2}(n,\kappa,\alpha,\Delta)L^\kappa(r)}\iint\displaylimits_{\big(\Delta((t+h)^{1/n})\setminus \Delta(t^{1/n})\big)^2}\hspace{-5mm}B^\kappa(r||\widetilde{x}-\widetilde{y}||)dxdy$$

$$=\lim_{r \to \infty}\frac{\kappa!r^{2n}}{r^{2n-\kappa\alpha}c_{2}(n,\kappa,\alpha, \Delta)L^\kappa(r)}\iint\displaylimits_{\big(\Delta( (t+h)^{1/n})\setminus \Delta( t^{1/n})\big)^2}\frac{L^\kappa(r||\widetilde{x}-\widetilde{y}||)}{(r||\widetilde{x}-\widetilde{y}||)^{\kappa\alpha}}d\widetilde{x}d\widetilde{y} $$
$$= \frac{\kappa!}{c_{2}(n,\kappa,\alpha, \Delta)}\iint\displaylimits_{\big(\Delta((t+h)^{1/n})\setminus \Delta( t^{1/n})\big)^2}{||\widetilde{x}-\widetilde{y}||^{-\kappa\alpha}}d\widetilde{x}d\widetilde{y}=\frac{|\Delta|^2h^2\kappa!}{c_{2}(n,\kappa,\alpha, \Delta)}E(||U-V||^{-\kappa\alpha}).$$
\end{proof}
 
\begin{remark}
The integrand in $\int_{\Delta}\int_{\Delta}||x-y||^{-\kappa\alpha}dxdy$ is a potential kernel, which also belongs to the class of unbounded generalized positive definite functions, see {\rm{\cite{phillips2019extension}}}. The methodology developed in this paper can be useful in studying properties of potential kernel and generalized positive definite functions, in particular, for obtaining their upper bounds and comparison, see~\cite{gaal2018integral}.

\end{remark}

Section 3 proved that integral functionals of  nonlinear transformations of Gaussian random processes over intervals in $\mathbb{R}_{+}$ converges to  processes with stationary increments. However, in the multidimensional case $\mathbb{R}^n,\ n>1,$ it might be not always true as the next example shows. In this example, the classical case of a disk observation window  is studied.

\begin{example}
\label{example1}
 Let the observation window $\Delta(t)\in\mathbb{R}^2, t>0,$ is a centered disk $B_2(t) = \{x\in\mathbb{R}^2: ||x||\leq t\}$ which is a homothetic transformation of a unit radius centered disk $B_2 = \{x\in\mathbb{R}^2: ||x||\leq 1\}.$  We will show that the limit process $Y(t), \ t\in[0,1],$ does not have stationary increments.

By Lemma \ref{lemma2} the variance of the increment $Y(t+h)-Y(t)$ is given by  the formula
$$\frac{\kappa!|B_2|^2h^2}{c_{2}(2,\kappa,\alpha, B_2)}E(||U-V||^{-\kappa\alpha}),$$ where $E(||U-V||^{-\kappa\alpha})$ is a function of $t$ and the expectation is taken over the set $B_2((t+h)^{1/2})\setminus B_2(t^{1/2}).$

Let us investigate the function $M(t,h)=E(||U-V||^{-\kappa\alpha})$ which  equals
\begin{equation}
\label{eq1}
M(t,h)= \frac{1}{|B_2((t+h)^{1/2})\setminus B_2(t^{1/2})|^2}  
\iint\displaylimits_{\big(B_2((t+h)^{1/2})\setminus B_2(t^{1/2})\big)^2}||x-y||^{-\kappa\alpha}dxdy.
\end{equation}

We will demonstrate that the derivative of $E(||U-V||^{-\kappa\alpha})$ is not identically equal~$0$ on the interval $t\in[0,1].$ 

Notice that for each fixed $t>0$ the velocity $v(x,t)$  is the same for all circle points $x\in \partial (B_2(t)).$ Therefore, the random variable $Y$ in Theorem~\ref{croft} is uniformly distributed on $\partial (B_2(t)).$ 

By applying Crofton's mean value formula to the set $B_2((t+h)^{1/2})\setminus B_2(t^{1/2})$ twice, the first time with a fixed external boundary and the second time with the fixed internal boundary, one obtains
$$\frac{d}{dt}M(t,h) = 2\frac{dV}{V}(M^{+}(t,h)-M^{-}(t,h)),$$ where $V = |B_2((t+h)^{1/2})\setminus B_2(t^{1/2})|$ and   $$M^{+}(t,h) = \frac{1}{|\partial (B_2(t+h)^{1/2})|}\int_{\partial (B_2(t+h)^{1/2})}\int_{B_2((t+h)^{1/2})\setminus B_2(t^{1/2})}||x-y||^{-\kappa\alpha}dxdy,$$ 
 $$M^{-}(t,h) = \frac{1}{|\partial (B_2(t)^{1/2})|}\int_{\partial (B_2(t)^{1/2})}\int_{B_2((t+h)^{1/2})\setminus B_2(t^{1/2})}||x-y||^{-\kappa\alpha}dxdy.$$

Let us consider the asymptotic behavior of $\frac{d}{dt}M(t,h)$ at the origin, i.e. $t=0,$ by finding the asymptotic behaviors of $M^{+}(t,h)$ and $M^{-}(t,h)$.

By the change of variables $ \tilde{x} = (t+h)^{1/2}x, \   \tilde{y} = (t+h)^{1/2}y,$
$$\lim_{t\to0} M^{+}(t,h) = \lim_{t\to0} \frac{(t+h)^{1-\kappa\alpha/2}}{2\pi\sqrt{t+h}}\int_{\partial (B_2(1))}\int_{B_2(1)\setminus B_2((\frac{t}{t+h})^{1/2})}||\tilde x- \tilde y||^{-\kappa\alpha}d\tilde xd\tilde y$$

$$=\frac{h^{1/2 - \kappa\alpha/2}}{2\pi}\int_{\partial (B_2(1))}\int_{B_2(1)}||\tilde x- \tilde y||^{-\kappa\alpha}d\tilde xd\tilde y.$$

Similarly, for $M^{-}(t,h)$ we get
$$\lim_{t\to0} M^{-}(t,h) = \lim_{t\to0}\frac{t^{1/2}(t+h)^{1/2}}{2\pi t^{1/2}}\int_{\partial (B_2(1))}\int_{B_2(1)\setminus B_2((\frac{t}{t+h})^{1/2})}||(t+h)^{1/2}\tilde x- t^{1/2}\tilde y||^{-\kappa\alpha}d\tilde xd\tilde y$$
$$=\frac{h^{1/2-\kappa\alpha/2}}{2\pi}\int_{\partial (B_2(1))}\int_{B_2(1)}||\tilde x||^{-\kappa\alpha}d\tilde xd\tilde y.$$

Thus,
$$\lim_{t\to0}\frac{d}{dt}M(t,h) = \frac{h^{1/2-\kappa\alpha/2}}{2\pi} \int_{\partial (B_2(1)}\bigg(\int_{B_2(1)}||\tilde x- \tilde y||^{-\kappa\alpha}d\tilde x - \int_{B_2(1)}||\tilde x||^{-\kappa\alpha}d\tilde x\bigg)d\tilde y$$

$$=\frac{h^{1/2-\kappa\alpha/2}}{2\pi} \int_{\partial (B_2(1)}\bigg(\int_{B_2(1)+ \tilde y}||\tilde x||^{-\kappa\alpha}d\tilde x - \int_{B_2(1)}||\tilde x||^{-\kappa\alpha}d\tilde x\bigg)d\tilde y.$$

For all $\tilde{y} \neq 0$ 
$$\int_{B_2(1)+\tilde{y}}||\tilde x||^{-\kappa\alpha}d\tilde x < \int_{B_2(1)}||\tilde x||^{-\kappa\alpha}d\tilde x$$ because the integration is over a non-centered disk $B_2(1) + \tilde y$.

Thus, $\lim_{t\to0}\frac{d}{dt}M(t,h) < 0$ and the function $Var(Y(t+h)-Y(t))$ is strictly decreasing in the neighborhood of the origin.
\end{example}

\begin{remark}\label{rem2}
For the case when the center of homothety $x_c$ is different from the center of the disk the increments are also non-stationary. Indeed, it is easy to demonstrate that $\lim_{t\to0}\big(M^{+}(t,h) - M^{-}(t,h) \big) \neq 0.$ 

Note that in this case, contrary to Example~\ref{example1},  for a fixed $t>0$ the velocity $v(x,t)$ varies over circle points $x\in \partial (B_2(t)).$ Therefore,  the random variable $Y$ in Theorem~\ref{croft} is not uniformly distributed on $\partial (B_2(t)).$  However, for any $t>0$ the density of $Y$ is the same for points $x \in \partial (B_2(t))$ belonging to same ray starting at $x_c.$ 

When $t\to0$ one has to compare the averages of $||x-y||^{-\kappa \alpha}$ over all points $y\in B_2(h^{1/2})$ for two points:

\begin{itemize}
\item [1)] the center of homothety $x=x_c\in B_2(h^{1/2}),$
\item [2)] a point on the boundary $x = x_b \in \partial B_2(h^{1/2}).$
\end{itemize}
As for each $x_b$ the homothety center $x_c$ is a midpoint of a symmetric arc, see Figure~{\rm{\ref{fig:arc}}}, then, by the symmetry, the averages over $y\in A$ of the distances $||x_c-y||^{-\kappa\alpha}$ and $||x_b-y||^{-\kappa\alpha}$ are equal. However, the average over $y\in B_2(h^{1/2})\setminus A$ of 
$||x_c-y||^{-\kappa\alpha}$ is greater than that of $||x_b-y||^{-\kappa \alpha},$ see Figure {\rm{\ref{fig:arc}}}.

So, $\lim_{t\to0}(M^{+}(t,h)-M^{-}(t,h))<0.$
\begin{figure}[h!]
  \centering
  \includegraphics[width=0.5\linewidth,trim={0 2cm 0 2cm},clip]{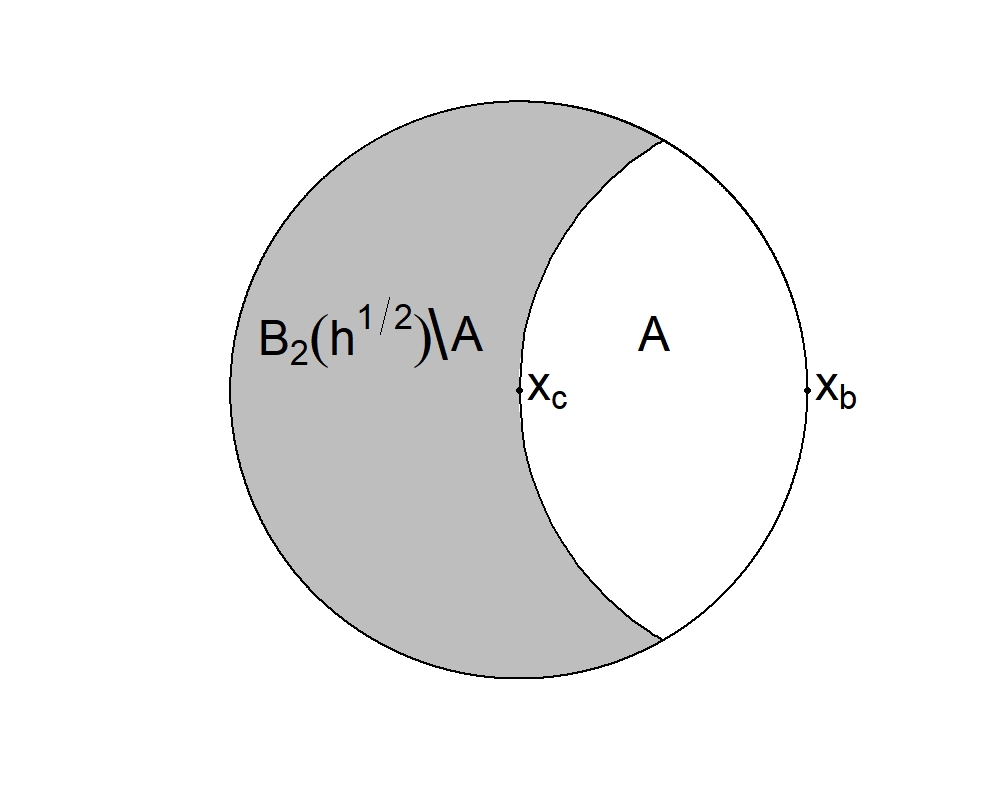}
  \caption{Regions of equal and different averages.}
  \label{fig:arc}
\end{figure}
\end{remark}

\begin{remark}
If the observation window $\Delta(t)\in\mathbb{R}^n,\ n\geq 2, \ t>0,$  is a homothetic transformation of a unit radius (not necessary centered) $n$-dimentional ball $B_n = \{x\in\mathbb{R}^n: ||x-a||\leq 1, \ ||a|| < 1\}$  with the parameter of homothety $t$, then by repeating the above reasoning it is straightforward to show the limit process $Y(t), \ t\in[0,1],$ does not have stationary increments. 
\end{remark}

Example \ref{example1} and Remark~\ref{rem2} show that for the disk $B_2(t)$ the averaging leads to the limit processes with a non-stationary structure of increments. If the observation window $\Delta$ has an arbitrary shape than one needs a detailed investigation of the difference of averages $M^{+}(t,h) - M^{-}(t,h),$ which is not straightforward for general domains. 

The following theorem is the main result of this paper. It shows that in the multidimensional case  increments of the limit process $Y(t)$  are non-stationary for any non-degenerate observation window $\Delta \in \mathbb{R}^n, \ n\geq1.$

\begin{theorem} Let the conditions of Theorem {\rm{\ref{leonenko}}} be satisfied. Then, for all sets $\Delta \in \mathbb{R}^n, \ n >1,$ the limit processes $Y(t), \ t\in[0,1],$ have non-stationary increments.
\end{theorem}
\begin{proof} 

By Lemma \ref{lemma2} and \eqref{eq1} 
$$Var(Y(t+h)-Y(t)) = \frac{\kappa!}{c_{2}(n,\kappa,\alpha,\Delta)}\iint\displaylimits_{\big(\Delta((t+h)^{1/n})\setminus \Delta(t^{1/n})\big)^2}||x-y||^{-\kappa\alpha}dxdy.$$

We prove that $Var(Y(t+h)-Y(t))$ is not constant in $t$ by a contradiction showing that the function   
$$I(t,h) = \iint\displaylimits_{\big(\Delta((t+h)^{1/n})\setminus \Delta(t^{1/n})\big)^2}||x-y||^{-\kappa\alpha}dxdy$$ is not constant with respect to $t.$ 

Let $I(t,h)$ do not depend on $t,$ i.e. $I(t,h) \equiv const(h).$ The process $Y(t)$ is defined on $[0,1],$ but without loss of generality we can consider $I(t,h)$ for $t\geq0.$  Indeed, by changing variables one gets
$$I(t,h)= (t+h)^{2-\kappa\alpha/n} \iint\displaylimits_{\big(\Delta(1)\setminus \Delta((\frac{t}{t+h})^{1/n})\big)^2}||x-y||^{-\kappa\alpha}dxdy$$ 
$$ = (t+h)^{2-\kappa\alpha/n} \iint\displaylimits_{\big(\Delta((1-\frac{t}{t+h})^{1/n}))\setminus\Delta(0)\big)^2}||x-y||^{-\kappa\alpha}dxdy = h^{2-\kappa\alpha/n} \iint\displaylimits_{\big(\Delta(1)\big)^2}||x-y||^{-\kappa\alpha}dxdy,$$ where the last equality follows from $I(t,h) \equiv const(h), \ t \in [0,1].$

Let $A_{t,h} = \Delta\big((t+h)^{1/n}\big)\setminus\Delta\big(t^{1/n}\big).$ Then we can change the variables in $I(t,h)$ as
$$I(t,h) = \iint\displaylimits_{A_{t,h}^2}||x-y||^{-\kappa\alpha}dxdy=\int_{\mathbb{R}^n}\int_{\mathbb{R}^n}I_{A_{t,h}}(x)I_{A_{t,h}}(y)||x-y||^{-\kappa\alpha}dxdy$$  
$$ = \int_{\mathbb{R}^n}I_{A_{t,h}}(y)\int_{\mathbb{R}^n}||x||^{-\kappa\alpha}I_{A_{t,h}}(x+y)dxdy,$$ where $I_A(\cdot)$ is a characteristic function of a set $A.$

Note, that as the origin is an interior point of $\Delta(1)$ there is a ball $B_n(\delta)$ of a radius $\delta>0$ that belongs to $\Delta(1),$ i.e. $B_n(\delta) \subset \Delta(1).$

For $t=0$ we obtain the following lower bound on $I(t,h)$ as
$$I(0,h)\geq \int_{\mathbb{R}^n}I_{B_n(\delta h^{1/n})}(y)\int_{\mathbb{R}^n}||x||^{-\kappa\alpha}I_{B_n(\delta h^{1/n})}(x+y)dxdy =\int_{\mathbb{R}^n}I_{B_n(\delta h^{1/n})}(y)$$
$$\times\int_{B_n(\delta h^{1/n}) - y}\hspace{-1mm}||x||^{-\kappa\alpha}dxdy\geq \frac{1}{2^n} \int_{\mathbb{R}^n} I_{B_n(\delta h^{1/n})}(y) \int_{B_n(\delta h^{1/n})}\hspace{-1mm}||x||^{-\kappa\alpha} dx dy,$$ as for any $y\in B_n(\delta h^{1/n})$   the shifted ball $B_n(\delta h^{1/n})-y$ always contains at least $2^{-n}$ of the original $B(\delta h^{1/n})$.

Hence, using the spherical change of coordinates we get
\begin{equation} \label{eq2}
I(0,h)\geq \frac{1}{2^n}|B_n(\delta h^{1/n})| \int_{0}^{\delta h^{1/n}} r^{n-1-\kappa\alpha}dr = Ch^{2-\frac{\kappa\alpha}{n}},
\end{equation}  where $C$ is a constant that does not depend on $h.$

Now we obtain the upper bound on $I(t,h).$ For any $C>0$
$$I(t,h) = \int_{\mathbb{R}^n}I_{A_{t,h}}(y)\bigg[ \int\displaylimits_{(A_{t,h}-y) \cap B_n(C)}||x||^{-\kappa\alpha}dx + \int\displaylimits_{(A_{t,h}-y) \cap \overline{B_n(C)}} ||x||^{-\kappa\alpha}dx\bigg]dy,$$ where $\overline{B_n(C)} = \mathbb{R}^n\setminus B_n(C).$

As for each $h$ and $y$ the volume $|(A_{t,h}-y) \cap B_n(C)| \to 0,$ when  $t\to \infty,$ we get
$$\int_{\mathbb{R}^n}I_{A_{t,h}}(y) \int\displaylimits_{(A_{t,h}-y) \cap B_n(C)} ||x||^{-\kappa\alpha}dxdy\to 0.$$

The second integral can be estimated as 
$$\int\displaylimits_{(A_{t,h}-y) \cap \overline{B_n(C)}} ||x||^{-\kappa\alpha}dx\leq C^{-\kappa\alpha}|(A_{t,h}-y)\cap \overline{B_n(C)}| \to C^{-\kappa\alpha}h|\Delta(1)|,\ \mbox{when}\ t \to \infty.$$

Let us choose $C = h^{\frac{1}{n}-\frac{\varepsilon}{\kappa\alpha}},$ where $\varepsilon\in (0,\frac{\kappa\alpha}{n}).$ Then 
\begin{equation} \label{eq3}
\int_{\mathbb{R}^n}I_{A_{t,h}}(y)\int_{(A_{t,h}-y) \cap \overline{B_n(C)}}||x||^{-\kappa\alpha}dxdy \to h^{2-\frac{\kappa\alpha}{n}+\varepsilon}|\Delta(1)|^2,\ t \to \infty.
\end{equation}  

Comparing \eqref{eq2} and \eqref{eq3}, we get that for sufficiently large $t$ it holds
$$I(t,h)\leq h^{2-\frac{\kappa\alpha}{n}+\varepsilon}|\Delta(1)|^2$$ and
$$Ch^{2-\frac{\kappa\alpha}{n}}\leq I(0,h) = I(t,h) \leq |\Delta(1)|^2h^{2-\frac{\kappa\alpha}{n}+\varepsilon}.$$
As $h$ can be selected arbitrary small, we get a contradiction.
\end{proof}

\section{ Numerical example}

This section presents three numerical examples showing the variances of increments $Y(s+h)-Y(s)$ in one and two-dimensional cases. We consider the most common cases in the literature, when $\Delta(1)$ is a one-dimensional interval [0,1], a two-dimensional disk and a square, and the Hermite rank $\kappa=1$. R code used for the numerical example is available in the folder "Research materials'' from \url{https://sites.google.com/site/olenkoandriy/}

For numerical calculations, the two-dimensional integrals   
$$E(Y(s+h)-Y(s))^2 = C\int_{\mathbb{R}^{2}} ||\lambda||^{(\alpha-2)}\left(\,\int\limits_{\Delta((s+h)^{1/2})\setminus \Delta(s^{1/2})}e^{i(\lambda,x)}\right)^2dx d\lambda$$  were approximated by the sums of the form
\begin{equation}
\label{eq5}
\sum_{i=0}^{m-1}\sum_{j=0}^{m-1}\mathscr{F}^2\bigg(\lambda_i^{(1)}, \lambda_j^{(2)}\bigg)\big((\lambda_i^{(1)})^2+(\lambda_j^{(2)})^2\big)^\frac{\alpha-2}{2}\big(\lambda_{i+1}^{(1)}-\lambda_{i}^{(1)}\big)\big(\lambda_{j+1}^{(2)}-\lambda_{j}^{(2)}\big),
\end{equation} where  $\mathscr{F}$ denotes the $2\mbox{-}$dimensional Fourier transformation of the indicator of the set $\Delta((s+h)^{1/2})\setminus{\Delta(s^{1/2})}$ and $(\lambda_i^{(1)},\lambda_j^{(2)}), \ i,j =0,...,m,$ form a grid of $m^2$ equidistant points in~$\mathbb{R}^2.$ 

The Fourier transformations of indicators of the disk and the square have explicit forms in terms of elementary functions that allow easy computations of the sums~\eqref{eq5}.

For the difference of disks with radiuses $(s+h)^{1/2}$ and $s^{1/2}$ the Fourier transform is
$$\frac{J_{1}(||\lambda||(s+h)^{1/2})-J_{1}(||\lambda||s^{1/2})}{||\lambda||^{1/2}},$$
where $J_{1}(\cdot)$ is the Bessel's function of the first kind of order 1.

For the difference of squares $\{(x_1,x_2) : |x_i| \leq (s+h)^{1/2}, i=1,2 \}$ and $\{(x_1,x_2) : |x_i| \leq s^{1/2} , i=1,2\}$ the Fourier transform of its indicator is
$$\frac{\sin((s+h)^{1/2}\lambda^{(1)})\sin((s+h)^{1/2}\lambda^{(2)})-\sin(s^{1/2}\lambda^{(1)})\sin(s^{1/2}\lambda^{(2)})}{\lambda^{(1)}\lambda^{(2)}}.$$

Figure \ref{fig:var} shows  variances of $Y(s+0.02)-Y(s)$ with the $s$-step $0.02$ for the following observation windows:
one-dimensional interval, two-dimensional disk and square.
\begin{figure}[h!]
  \centering
  \includegraphics[width=0.7\linewidth,trim={0 5mm 0 26mm},clip]{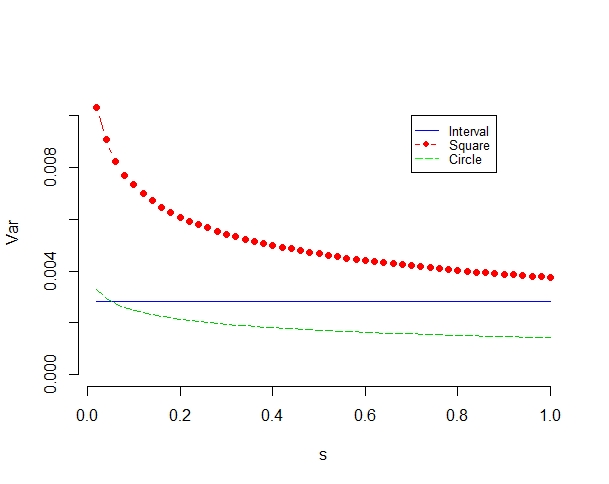}\vspace{-4mm}
  \caption{$Var(Y(s+0.02)-Y(s))$ for different observation windows.}
  \label{fig:var}
\end{figure}

The numerical example show that for the case $n=2$ variances of the increments of the limit process are not equal for different sets $\Delta$ and decreases when $s$ increases, while in the one-dimensional case $n=1$ variances of the increments are constant.
\section{ Conclusions and future studies}
It was shown that contrary to the classical one-dimensional case of stochastic processes, for any choice of an observation window,  integral functionals of nonlinear transformations of long-range dependent random fields converge to the generalized Hermite-type process with non-stationary increments.

In the future studies, it would be interesting to investigate:

- the case of weighted integral functionals, see \cite{ alodat2020, anh2019lse,  ivanov1989statistical};

- the case of filtered random fields, see \cite{alodat2019limit};

- whether there exists a normalization  dependent on $t$ that can result in a limit process with stationary increments;

- other random processes and fields that require using $\iint_{\Delta^2} g(||x-y||)dxdy$ instead of $\iint_{\Delta^2}||x-y||^{-\kappa\alpha}dxdy$ for some suitable functions $g(\cdot)$;

- application of the approaches to integrals of generalized positive-definite functions \cite{gaal2018integral, phillips2019extension} and average distances \cite{baddeley1977integrals, nickolas2011distance}.\\

\noindent\textbf{Acknowledgement} A.Olenko was partially supported under the Australian Research Council Discovery Projects funding scheme (project number DP160101366). The authors are also grateful to the  anonymous referees for their suggestions that helped to improve the style of the paper.

\providecommand{\bysame}{\leavevmode\hbox to3em{\hrulefill}\thinspace}
\providecommand{\MR}{\relax\ifhmode\unskip\space\fi MR }
\providecommand{\MRhref}[2]{%
  \href{http://www.ams.org/mathscinet-getitem?mr=#1}{#2}
}
\providecommand{\href}[2]{#2}

\end{document}